\documentclass{amsart} 

\usepackage{pgfplots}
\pgfplotsset{compat=1.15}
\usepackage{mathrsfs}
\usetikzlibrary{arrows}

\usepackage[utf8]{inputenc}
\usepackage{a4}
\usepackage{graphicx, xcolor}

\usepackage{hyperref}
\hypersetup{colorlinks=false, linkcolor=blue}

\usepackage{amsmath,amssymb}
\newtheorem{definition}{Definition}
\newtheorem{theorem}{Theorem}
\newtheorem{example}{Example}

\date{}

\numberwithin{equation}{section}

\usepackage{algorithm}
\usepackage{algpseudocode}

\usepackage{cite}

\newcommand{\V}{\vartheta}

\title{Quantifier Elimination for Normal Cone Computations}

\author{Michael Mandlmayr}
\address{[M.M. \& A.K.U.] Johann Radon Institute for Computational and Applied Mathematics, Austrian Academy of Science, Altenbergerstraße 69, A-4040 Linz, Austria}
\email{michael.mandlmayr@live.com}
\email{akuncu@ricam.oeaw.ac.at}

\author[Ali K. Uncu]{Ali K. Uncu}
\address{[A.K.U.] University of Bath, Faculty of Science, Department of Computer Science, Bath, BA2 7AY, UK}
\email{aku21@bath.ac.uk}

  

\begin{document}

\begin{abstract}
    We present effective procedures to calculate regular normal cones and other related objects using quantifier elimination. This method of normal cone calculations is complementary to computing Lagrangians and it works best at points where the constraint qualifications fail and extra work for other methods becomes inevitable. This method also serves as a tool to calculate the regular co-derivative for semismooth* Newton methods.  We list algorithms and their demonstrations of different use cases for this approach.
\end{abstract}

\maketitle


\section{Introduction}\label{sec:intro}


The (regular) normal cone mapping is among the most important objects in constraint optimization problems. The (regular) normal cone contains information about the constraints that are essential for first-order optimality conditions.  In this paper we want to explore the possibilities of computing the \textit{regular normal cone} (mapping) by means of quantifier elimination. We are interested in optimization problems presented in the form
\begin{equation}\label{eq:minimize}\min_{x\in C}f(x),\end{equation} where typically $ C=\{x\, :\, g(x)\in D\}\subset \mathbb{R}^n$ for smooth functions $g:\mathbb{R}^n\to\mathbb{R}^s$, $f:\mathbb{R}^n\to\mathbb{R}$ and a closed set $D\subset \mathbb{R}^s$. Classically, the method of multipliers is used to solve such problems. The method of multipliers may also appear under the name Lagrange function or Karush-Kuhn-Tucker (KKT) conditions. KKT approach works under so called constraint qualifications \cite{benko2016numerical}.
This method involves the normal cone, formally defined as follows.
\begin{definition}[Definition 6.3, pg 199, \cite{RockWets98}]
\label{def_oldRoc}
Let $C\subset \mathbb{R}^n$ and $\bar{x}\in C$. A vector $v$ is normal to $C$ at $\bar{x}$ in the regular sense, or $v$ is a regular normal, written $v\in\hat{N}_C(x)$, if 
\begin{equation}
    \langle v,x-\bar{x}\rangle\le o(||x-\bar{x}||),\;\text{for }x\in C,
\end{equation}
where $\langle \cdot,\cdot\rangle$ and $||\cdot||$ are Euclidean inner product and norm, respectively.
\end{definition}

$\hat{N}_C(\bar{x})$ is given in an implicit (and a slightly unclear) way in Definition~\ref{def_oldRoc}. Intuitively, the normal cone to the set $C$ at a point $\bar{x}\in C$ is the collection of all vectors $v$, that has a non-positive scalar product with all directions, i.e. $x-\bar{x}$, that locally stay in $C$.

The KKT conditions come from a more general stationarity concept. We call $x\in C$ a \textit{stationary solution} of \eqref{eq:minimize} if
$$- \nabla f(x) \in\hat{N}_C(x)$$ is satisfied.
Solving this inclusion directly is impractical, instead one solves
\begin{equation}\label{eq:KKT_intro_condition}-\nabla f(x)= \nabla g(x)^T\lambda\end{equation} for  $x\in C$ and
$\lambda\in N_D(g(x)),$ called the \textit{Lagrange multiplier}, where \begin{equation*}
N_D(\bar{d})=\limsup_{d\underset{D}{\rightarrow}\bar{d}}\hat{N}_D(d).
\end{equation*} $N_D(\bar{x})$ is commonly referred to as the \textit{limiting normal cone}.
The condition \eqref{eq:KKT_intro_condition} is necessary for optimality if 
\begin{equation}\label{eq:const_qual}\hat{N}_C(x)\subset\nabla g(x)^T  N_D(g(x))\end{equation}
holds.


When the \eqref{eq:const_qual} inclusion does not hold, the KKT-conditions can fail. Consequently, we might not be able to detect solutions. These issues are present even in a simple example, such as the minimization problem: 
\begin{align}
  \nonumber  \min&\; x+y\\
  \label{conditionsKKT}  \text{subject to:}&\\
  \nonumber x&\ge 0\\
  \nonumber  (y+x^2)(y-x^2)&=0.
\end{align}
Computing the regular normal cone $\hat{N}_C(\bar{x})$, where $C = \{(x,y)\in\mathbb{R}^2\, :\, x\geq 0 \land (y+x^2)(y-x^2)=0\}$, using the KKT conditions is not directly possible. 
In this problem, we can deduce that $y=x^2$ or $y=-x^2$ from the second constraint. If  $y=x^2$ then the objective to minimize along this curve is $x+x^2$ and since x is positive the minimum is attained at $(0,0)$. Similarly, if $y=-x^2$ then the objective along this curve is $x-x^2$ which has a local minimum at $(0,0)$ since $x$ is positive. Consequently, $(0,0)$ is a local minimizer for the original problem. We then would check the stationarity by means of the KKT conditions.
Doing so by computing the Lagrangian  at $(0,0)$, $\mathcal{L}_{\lambda_1,\lambda_2}(0,0)$, yields the following:
\begin{equation*}
    \mathcal{L}_{\lambda_1,\lambda_2}(0,0)= \underbrace{(1,1)^T}_{\nabla f(0,0)}+ \underbrace{(0,0)^T\lambda_1+ (-1,0)^T\lambda_2}_{\nabla g(0,0)^T \lambda}\overset{?}{=}(0,0)^T
\end{equation*} for any $\lambda_2 \ge 0$ and $\lambda_1\in \mathbb{R}$.
However, the equation in question cannot be satisfied.

\definecolor{ccqqqq}{rgb}{0.8,0,0}
\definecolor{qqwwzz}{rgb}{0,0.4,0.6}
\definecolor{qqzzcc}{rgb}{0,0.6,0.8}
\definecolor{qqttcc}{rgb}{0,0.2,0.8}

In this example $g:\mathbb{R}^2\to \mathbb{R}^2$ is defined by $g(x,y)=(-x,(y+x^2)(y-x^2))$ and $D=\left[0,\infty\right)\times \{0\}$.
In Figure~\ref{fig:constrfail}, the light-blue area corresponds to ${\color{qqzzcc}\hat{N}_C(0,0)} = \{(x,y)\in\mathbb{R}^2\, :\, x\leq 0 \}$,  and the red ray, ${\color{ccqqqq}\nabla g(0,0)^T\hat{N}_D(0,0)}$, is a lower estimate of the regular normal cone used in the KKT conditions. The blue {arrow} is the negative gradient of the objective ${\color{qqttcc}-\nabla f(0,0)}$. We can observe that the estimate from the KKT conditions does not contain the negative gradient, while the actual regular normal cone does. This means that $(0,0)$ is a stationary solution that goes undetected by the KKT conditions.

\begin{figure}[h]
\includegraphics{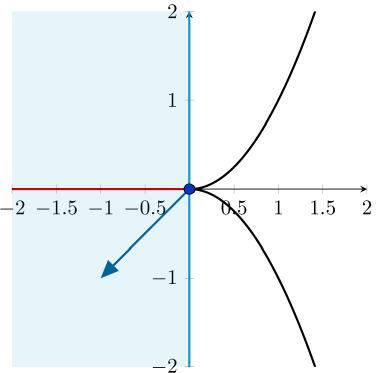}
    \caption{Example of KKT conditions failing to identify the entire regular normal cone calculated at $(0,0)$, the blue shaded region.}
    \label{fig:constrfail}
\end{figure}

One way of overcoming the obstacle of the failing KKT conditions is, informally speaking, shifting the complexity from the mapping $g$ to the domain $D$. 
To that end, a possible approach is to use disjunctive programming \cite{FleKanOut07,Gfr14a,St03,balas1979disjunctive}. Instead, we will focus on the direct computation of the regular normal cone from its set description. For optimization problems that can be described with polynomial constraints, i.e. when the normal cone has a semi-algebraic description, we will be able to use quantifier elimination. 

The normal cone also appears as a tool of linearization in a recently developed class of Newton-type methods for set-valued mappings. This semismooth* Newton method has a very light smoothness assumption (semismoothness*) and it applies to generalized equations. Moreover, simpler linearizations are available for the big class of subspace contained derivative mappings \cite{Mandlmayretal}, but in the most general case the normal cone has to be computed. The semismooth* Newton method is rather new and covered in \cite{MandlmayrThesis,gfrerer2021semismooth,gfrerer2022application,gfrerer2022local}. 

In this paper, we explore the possibilities of computing the normal cone of a semi-algebraic set $C$ by means of quantifier elimination (QE). This will enable us to overcome problems, when constraint qualifications are violated and to automatically compute co-derivatives. In particular, we will focus on simplifying these computations by a divide-and-conquer approach naturally induced by the cylindrical algebraic decomposition (CAD) of $C$. We will be able to confirm if a given vector is in the normal cone $\hat{N}_C(x)$ and calculate the normal cone exactly without the need for calculus rules.


CAD is limited to finite-dimensional semi-algebraic sets $C$. It is most applicable if the constraints impose a product structure with low sub-dimensions. If the dimension is too high, the complexity of the required algorithms is too high. As a side note, in \cite{anai2003convex} CAD was used directly to solve some semi-definite programming problems where all constraints are linear. Optimization problems with polynomial objective functions and constraints belong to tame optimization \cite{kaka, Io09}. Feeding tame optimization problems as a whole can also be attempted but success will depend on the termination of the CAD algorithm. Our approach is not limited to tame optimization.

We compute the normal cone and consequently, the co-derivatives using QE/CAD. Our goal is to make the computations of co-derivatives more accessible.
Both constraint optimization and semismooth* Newton method involve tools from variational analysis; further information on this topic can be found in \cite{RockWets98, Io07, Mo18}. Applying these symbolic computation paradigms to compute the normal cone has, to the best of our knowledge, not been done before.

The outline of this paper is as follows. In Section~\ref{sec:backgroundOpti}, we recall some basic principles from variational analysis, optimization, and the semismooth* Newton methods. In Section~\ref{sec:QE}, we recall a background on Quantifier Elimination and Cylindrical Algebraic Decomposition. 
Section~\ref{sec:SemialgebraicDesc} has the presentation of the semi-algebraic description and the quantified structure that defines normal cones, later to be used in quantifier elimination. In Section~\ref{sec:main}, we sequentially present our main results. We give pseudo-algorithms of normal cone related questions that can arise in optimization and how one can attempt to solve these using quantifier elimination. These algorithms are followed by examples and discussions to demonstrate their use. In Section~\ref{sec:otherApplications}, we mention other possible uses of quantifier elimination in the context of optimization, in particular to problems related to semismooth* Newton method. Section~\ref{sec:out} has a list of future research directions that we plan to pursue.

Interested readers are encouraged to look at the \emph{Mathematica} worksheet that demonstrates computations that accompany the examples in this paper attached under the ancillary files of the ArXiv image of this article and on the second author's website \href{http://www.akuncu.com}{http://www.akuncu.com} under publications.

\section{Some Background on Optimization}\label{sec:backgroundOpti}

We start by defining the tangent cone. The tangent cone to a set $C$ at a point $\bar{x}\in C$ is the linearization of the set $C$ at $\bar{x}$ and generalizes the concept of tangents.
\begin{definition}
\label{tangentcone}
Let $C\subset \mathbb{R}^n$ and  $\bar{x}\in C$
\begin{equation*}
T_C(\bar{x}):=\{w\in \mathbb{R}^n:\exists x_n\underset{C}{\rightarrow}\bar{x}, \exists \alpha_n\downarrow 0:\frac{x_n-\bar{x}}{\alpha_n}\rightarrow w\}
\end{equation*}
is called the tangent cone of $C$ in $\bar{x}$.
\end{definition}
The normal cone can be defined independent of the tangent cone, as we have seen in Definition~\ref{def_oldRoc}. It can be viewed as the generalization of normal vectors. 
The tangent and (regular) normal cone are related by a concept called polarity.
\begin{definition}
\label{def:polar}
    For a given set $A$, we define its polar set as follows $$A^{\circ}:=\{v:\langle v,x \rangle \le 0 \;\forall x\in A\}.$$
\end{definition}
 Tangent and normal cones satisfy the following property \begin{equation}\label{eq:polarity}T_C(\bar{x})^{\circ}=\hat{N}_C(\bar{x}).\end{equation}

\subsection{Constraint Optimization}
In constraint optimization, we are interested in minimizing a given function $f:\mathbb{R}^n\to\mathbb{R}$ over a closed set $C\subset \mathbb{R}^n$ or compactly written, as
$$\min_{\bar{x}\in C} f(\bar{x}).$$
First-order optimality conditions basically state that no descent can locally be attained without exiting the domain $C$, this is a condition concerning the tangent cone and the gradient of the function at the point in question. Employing polarity \eqref{eq:polarity} this turns into a condition on the normal cone and the gradient.
\begin{theorem}\cite[Theorem 6.14]{RockWets98}\label{thm:1}
For a differentiable function $f:\mathbb{R}^n\to\mathbb{R}$ and a closed set $C\subset\mathbb{R}^n$ a necessary condition for $\bar{x}$ being locally optimal is
$$\langle \nabla f(\bar{x}),v\rangle\ge0\quad\forall v\in T_C(\bar{x}),$$
which is equivalent to
\begin{equation}
    -\nabla f(\bar{x})\in \hat{N}_C(\bar{x}).
\end{equation}
\end{theorem}

Now it is typically not practicable to compute $\hat{N}_C(\bar{x})$ directly by Definition~\ref{def_oldRoc}, however, there are calculus rules that hold under \textit{constrained qualifications}, i.e. conditions on the constraints which ensures that \eqref{eq:const_qual} holds, for the setting, when $C=\{x:g(x)\in D\}$, $g:\mathbb{R}^n\to \mathbb{R}^s$, and $D$ is a closed set. For example, in \textit{nonlinear programming} (optimization problems involving nonlinear equality and inequality constraints) it is customary to take $D= \prod_k^u \{0\}\times \prod_l^v \left[0,\infty\right)$, where $u+v=s$.
\begin{theorem}\cite[Theorem 6.14]{RockWets98}\label{thm:constraint_qual}
For a differentiable function $g:\mathbb{R}^n\to \mathbb{R}^s$ and a closed set $D$ we always have
\begin{equation*}
    \nabla g(\bar{x})\hat{N}_D(g(\bar{x}))\subset N_C(\bar{x}).
\end{equation*}
If, in addition, the following constraint qualification holds
\begin{equation}
\label{constrqual}
    \nabla g(\bar{x})\lambda =0\land \lambda\in N_D(g(\bar{x})) \implies \lambda=0 
\end{equation} then
\begin{equation*}
        N_C(\bar{x})\subset \nabla g(\bar{x})N_D(g(\bar{x}))
\end{equation*}
also holds.
\end{theorem}
The set $N_D(\bar{x})$ can be seen as a generalized version of Lagrange multipliers.

Here \eqref{constrqual} is a constraint qualification that is particularly satisfied if $\nabla g(\bar{x})$ has full rank. Similar constraint qualifications are also needed in the KKT conditions, which are special cases of Theorem~\ref{thm:1}.
For a set $C$ there are many ways to choose the pair: a function $g$ and a set $D$. Complexity can be shifted from one to the other. Classically, a very simple $D$ is chosen, often convex polyhedral but also other setups are very interesting, e.g. disjunctive programming \cite{disjunctive}.
The choice of $g$ and $D$ is crucial for fulfilling said constraint qualifications. At a given point these can be valid for one description without being valid for the other. 
There are many different constraint qualifications, where some have the upside of being very light (eg. generalized Abadie, Guignard constraint qualification \cite{benko2016numerical}) but impractical to check in applications, and others that are easy to check but more restrictive (eg. linear independence constraint qualification, Mangasarian Fromowitz constraint qualification \cite{benko2016numerical})

Lower estimates of $\hat{N}_C(\bar{x})$ are used as sufficient conditions for stationary solutions, while upper estimates are necessary ones.
In practice, identifying upper estimates is highly important because the necessary conditions imposed by these estimates yield criteria to check for identifying candidate solutions.

\subsection{Semismooth* Newton methods}

Recently, a new class of Newton-type methods for generalized equations has been developed. 
For a generalized equation we are given  set-valued mapping $F:\mathbb{R}^n \rightrightarrows\ \mathbb{R}^n$ and we are interested in finding a point $x\in\mathbb{R}^n$ fulfilling
\begin{equation*}
    0\in F(x).
\end{equation*}

In all Newton-like methods, some sort of linearization takes place. In the most general case, for the semismooth* Newton method, the construction of the linearization is based on the (regular/limiting) normal cone. For a broad class of subspace contained derivative practical simplifications have been made and result in a simpler construction. 

This new kind of Newton-like method differs from classical ones in two aspects.
An additional, typically cheap step is introduced - the \textit{approximation step}. The purpose of this step is to construct a point contained in the graph of the set-valued mapping with the properties that it stays reasonably close to the previous iterate and that the residue is as small as possible. We can interpret the classical Newton method for a smooth function $f:\mathbb{R}^n\to\mathbb{R}^n$ as a semismooth* Newton method. Here the approximation step is for a given iterate $x^{(k)}$ simply taking $(x^{(k)},f(x^{(k)}))\in \mathrm{gph}f$.

The second difference lies in the construction of a linearization, which in the most general case is based on the regular/limiting normal cone and is called regular/limiting co-derivative.

\begin{definition}\cite[Definition 8.32]{RockWets98}
The multi-function $\hat{D}^*F(\bar{x},\bar{y})(u):\mathbb{R}^n\rightrightarrows\mathbb{R}^m$ defined by 
\begin{equation*}
\hat{D}^*F(\bar{x},\bar{y})v^*:=\{u^*\in\mathbb{R}^n|(u^*,-v^*)\in \hat{N}_{\mathit{gph}F}(\bar{x},\bar{y})\},v^*\in \mathbb{R}^m 
\end{equation*}  
is called regular co-derivative  of $F$ at $(\bar{x},\bar{y})$.
\end{definition}
For a single-valued differentiable function, this, of course, corresponds to the Jacobian, we have
\begin{equation*}
\hat{D}^*F(\bar{x},F(\bar{x}))v^*=\nabla F(\bar{x})^Tv^*.
\end{equation*}
\begin{definition} \cite[Definition 8.32]{RockWets98}
The multifunction $D^*F(\bar{x},\bar{y})(u):\mathbb{R}^n\rightrightarrows\mathbb{R}^m$ defined by 
\begin{equation*}
D^*F(\bar{x},\bar{y})v^*:=\{u^*\in\mathbb{R}^n|(u^*,-v^*)\in N_{\mathit{gph}F}(\bar{x},\bar{y})\},v^*\in \mathbb{R}^m 
\end{equation*}  
is called limiting co-derivative  of $F$ at $(\bar{x},\bar{y})$.
\end{definition}

The beauty of these constructions is that they are still well-defined for set-valued mappings and also yield weaker smoothness requirements even for the single-valued case. In the case of nonlinear programming, when applying this framework to the Lagrangian, this abstract framework can be interpreted as an active-set strategy.

All of the information about the linearization is contained in the normal cone to the graph of the mapping in a given point, therefore it is again interesting to apply CAD to it. Done by hand these computations are long and difficult,
in the dissertation \cite{MandlmayrThesis} these normal cones were for a contact problem with Coulomb friction, see Example~\ref{ex:5}, constructed for every possible point and resulted in many pages of computations.

\section{Some Background on Quantifier Elimination}\label{sec:QE}

A \textit{Tarski formula} $\Phi(x_1,\dots, x_n)$ is an expression that involves polynomial relations $f\sigma\, 0$, where $\sigma\in\{=,\not=, >,\geq, <, \leq\}$ for $f\in\mathbb{Z}[x_1,\dots, x_n]$ (and $f:\mathbb{R}^n\rightarrow\mathbb{R}$), combined using Boolean connectives $\land$ and $\lor$. In general, we say that a formula is in the \textit{extended} Tarski language if it is possible to find a Tarski formula equivalent to it. For example, $\sqrt{x^2}>0$ is in the extended Tarski language, since this expression is equivalent to $x>0 \lor -x>0$. Sets that are defined by some formula in the extended Tarski language are called \textit{semi-algebraic}. A \textit{quantified formula} is a formula with added quantifiers on some variables that appear in these formulas. A \textit{prenex-quantified formula} is an expression of the form \[Q_1 x_{1,1}\dots x_{1,k_1} \dots Q_a x_{a,1}\dots x_{a,k_2}\, \Phi(x_1, \dots , x_{a,k_2}, y_1, \dots, y_n),\] where $Q_i\in\{\exists, \forall\}$. In the form above, the $x_i$s are called \textit{quantified} and $y_j$s are called \textit{free} variables. If all the variables are quantified then the formula is called a \textit{sentence}. Also, it is known that any quantified formula can be turned into a prenex-quantified formula. Hence, we will not be specializing between forms when it is not needed.

Obtaining the quantifier-free formula from a quantified one is called \textit{Quantifier Elimination} (QE). In 1951, Tarski proved that there is a quantifier-free equivalent formula for every quantified Tarski formula.  Quantifier elimination problems are known to arise many different fields such as economics \cite{Mulliganetal2018b}, mechanics \cite{Ioakimidis2019a}, mathematical biology \cite{RostSadeghimanesh2021b}, reaction networks \cite{RahkooySturm2021c}, AI to pass mathematical exams \cite{Wadaetal2016a}, and motion planning \cite{Wilsonetal2013a}. 

In 1975, Collins \cite{Collins1975} produced the \textit{cylindrical algebraic decomposition} (CAD) algorithm (and the synonymously named algebraic object). CAD decomposes the real space $\mathbb{R}^n$ into a finite number of disjoint semi-algebraic cells with a uniform property that are also ordered in a (cylindric) fashion that the projection of two cells onto lower dimensions (with regards to variable ordering used in the CAD calculations) are either exactly the same set or disjoint. CAD calculations require a variable ordering followed by some projections of the polynomials that appear in the formula, and then a lifting phase of the cells to the full dimension. Once a CAD is calculated for a problem, one can check representative points for the cells against the formula and deduct the validity of the formula in the cells. In particular, CAD can be used to perform QE on any problem that involves semi-algebraic sets defined by finitely many polynomial constraints.

It is known that QE and CAD have doubly-exponential worst-case complexity \cite{DavenportHeintz1988a, BrownDavenport2007}. To be precise, given $r$ polynomials with maximum degree $d$ in $n$ variables the worst-case complexity is $(rd)\char`\^( 2\char`\^ \mathcal{O}(n))$. However, we do not usually observe these complexities in real-world applications. Moreover, lowering the complexity of CAD (by means of different projection operations \cite{Brown2001c, lazard1994improved, brown2020enhancements}) and QE (by means of incomplete methods such as cylindrical algebraic coverings, virtual term substitutions, etc. \cite{Abrahametal2020b, Brown2005b, DavenportTonksUncu2021, DavenportTonksUncu2023}) calculations is a highly-active field of research \cite{Bradfordetal2021a}. Most modern computer algebra systems include implementations of CAD and other QE methods; some of these implementations even accept formulas in extended Tarski language. Therefore, we will also write things in the extended language for brevity.

\section{A semi-algebraic description}\label{sec:SemialgebraicDesc}
The most established definition of the normal cone involves some $o$-notation \cite{RockWets98}.
We would like to bring this definition into a form that is suitable for the application of CAD.
We can give an equivalent definition of the normal cone using the notation $B_{\delta}(\bar{x})=\{y\, :\,||\bar{x}-y||\le \delta\}$ rather than the $o$-notation of Definition~\ref{def_oldRoc}.
\begin{definition}\label{Def_original}
Let $\bar{x}\in C$, then $v^*\in \hat{N}_C(\bar{x})$ if
    \begin{equation}
        \forall \epsilon>0\, \exists\delta>0\, \forall x\in \left(B_\delta(\bar{x}) \setminus \{\bar{x}\} \right)\cap C\ \ \langle v^*,x-\bar{x}\rangle\le \epsilon||x-\bar{x}||
    \end{equation}
    is satisfied.
\end{definition}
We can slightly change how we represent this definition using a disjunction: 
\begin{definition}\label{Def_squared_norms}
  Let $\bar{x}\in C$, then   $v^*\in \hat{N}_C(\bar{x})$  if
  \begin{align}\label{eq:def3} \forall \epsilon>0\, \exists\delta>0\, \forall x\in &\{y\in C\setminus \{\bar{x}\} \ : \;||y-\bar{x}||^2\le \delta^2\}\ \ \langle v^*,x-\bar{x}\rangle^2\le \epsilon^2||x-\bar{x}||^2\lor\langle v^*,x-\bar{x}\rangle\le 0
    \end{align}
    holds.
\end{definition}
Typically normal cones are defined for closed sets $C$. Going forward, We will use $\hat{N}_C$ for the set of $v^*$ that satisfies \eqref{eq:def3} when $C$ is not closed.
While Definitions~\ref{Def_original} and \ref{Def_squared_norms} are the same, the latter one is more suitable for CAD/QE applications. At the very least one can see that the extra squares clears the square roots in the Euclidean norm $||\cdot||$ and slightly improves the expression by turning it into a Tarski formula.

Theoretically, if $C$ itself is semi-algebraic, we can calculate the normal cone $\hat{N}_C(\bar{x})$ using quantifier elimination. We demonstrate it in detail with the following example:

\begin{example}\label{example_1D} Let $C:= \{y: y\geq 0 \}\subset\mathbb{R}$. Let $\bar{x}\in C$, then $v^*\in \hat{N}_C(\bar{x})$ if
\begin{equation}\label{eq_ex_1D}\forall \epsilon>0\, \exists \delta>0 \, \forall x\in \{y\in C\, : y\not=\bar{x} ,\, (y-\bar{x})^2 <\delta^2 \}\ \ v^*(x-\bar{x}))^2 \leq \epsilon^2 (x-\bar{x})^2 \lor v^*(x-\bar{x})\leq 0\end{equation} is satisfied.
Here $C$ is clearly semi-algebraic. Therefore, the inner quantified variable $x$ is also defined through a semi-algebraic set (with defining relations $x\geq 0 \land x\not=0 \land (x-\bar{x})^2<\delta^2$). Variables $\epsilon,\ \delta$, and $x$ are quantified, whereas $\bar{x}$ and $v^*$ are quantifier-free. Hence, after quantifier elimination, we would get solutions in the space $(\bar{x},v^*)\in\mathbb{R}^2$. Applying QE to \eqref{eq_ex_1D} we get the equivalent quantifier-free formula 
\begin{equation}\label{eq_ex_1D_QE}\bar{x} < 0 \lor (\bar{x} = 0 \land v^* \leq 0) \lor (\bar{x} > 0 \land v^* = 0).\end{equation}
There are three distinct clauses in \eqref{eq_ex_1D_QE} that we should comment on. The innermost quantified formula of \eqref{eq_ex_1D} is equivalent to \[\forall x \left( x\geq 0 \land x\not=0 \land (x-\bar{x})^2<\delta^2 \Rightarrow (v^*(x-\bar{x}))^2 \leq \epsilon^2 (x-\bar{x})^2 \lor v^*(x-\bar{x})\leq 0\right)\] or equivalently
\[\forall x \left( x< 0 \lor x=0 \lor (x-\bar{x})^2\geq\delta^2 \lor (v^*(x-\bar{x}))^2 \leq \epsilon^2 (x-\bar{x})^2 \lor v^*(x-\bar{x})\leq 0\right),\] by $(P\Rightarrow Q) \Leftrightarrow (\lnot P \lor Q)$.
Note that this is true for all $x\leq 0$. When $\bar{x}<0$ and $x>0$, for every $\epsilon>0$, $\delta = x-\bar{x}>0$ makes the middle-clause above, $(x-\bar{x})^2 \geq \delta^2$, true. Hence, although $\bar{x}<0$ is not relevant to the actual question (when $\bar{x}\in C$), it is among the semi-algebraic sets that satisfy \eqref{eq_ex_1D}. When $\bar{x}$ is on the boundary of $C$, i.e. $\bar{x} = 0$, we see that the only solution is $v^*\leq 0$. Finally, for any interior point of $C$, i.e. $\bar{x}>0$, we get the only solution $v^* = 0$.
\end{example}

On a practical note, one can get rid of the irrelevant solutions to the QE problem by adding the clause $\bar{x}\in C$ in conjunction to the whole \eqref{eq:def3}. Quantifier elimination calculations benefit highly from preprocessing and simplifications. We get the fastest results when we focus on the cusp locations of our domains when the gradients and derivative-based arguments fail. In that sense, using QE in this context is also complementary to the well-established methods that utilize gradients.

\section{Main Results}\label{sec:main}

There are 3 questions relevant to normal cone mapping calculations of varying difficulty that we can attempt with the quantifier elimination. For some semi-algebraic $C\subset\mathbb{R}^n$, these problems are
\begin{itemize}
\item checking for a given point $\bar{x}\in C$ and candidate $v^*\in\mathbb{R}^n$ if $v^*\in \hat{N}_C(\bar{x})$ holds, which involves $n+2,$ variables.
    \item computing the normal cone at a fixed point at $\bar{x}\in C$, which involves $2n+2$ variables,
    \item computing the full normal cone mapping $\hat{N}_C(\bar{x})$, which involves $3n+2$ variables,
\end{itemize}
In each case, from Definition~\ref{Def_original}, there are $n+2$ quantified variables.


We start by giving the pseudo-algorithm for checking for a given point $\bar{x}$ and a candidate solution $v^*$, if the candidate solution is in $N_C(\bar{x})$.

\begin{algorithm}
\caption{Checking if a given $v^* \in \hat{N}_C(\bar{x})$}\label{alg:cap}
\begin{algorithmic}
\Require $C$ semi-algebraic\\\vspace{-3mm}
\State $S \leftarrow \forall \epsilon>0\,\exists\delta>0\,\forall x\in \{y\in C\setminus \{\bar{x}\}\, :\, ||y-\bar{x}||^2\le \delta^2\}\ \ \langle v^*,x-\bar{x}\rangle^2\le \epsilon^2||x-\bar{x}||^2\lor\langle v^*,x-\bar{x}\rangle\le 0$ 
\State $TF \leftarrow$ Apply QE to $S$.\\
\Return $TF$ \Comment $TF$ is True/False since $S$ is a sentence.
\end{algorithmic}
\end{algorithm}


        

\begin{example} We begin by checking that $v^* = (-1,-1) \in \hat{N}_C(0,0)$ for the KKT conditions to \eqref{conditionsKKT}. Quantifier elimination of \begin{align*}\forall\epsilon>0\exists\delta>0\forall x,y [x\geq 0 \land (y-x^2)(y+x^2)=0\land x\not=0&\land y\not=0 \land x^2+y^2 <\delta^2\\ &\Rightarrow ((-x-y)^2<\epsilon^2 (x^2+y^2) \lor -x-y\leq 0)]\end{align*} yields True and certifies that $(-1,-1)\in \hat{N}_C(0,0)$. Note that it was not possible to verify the stationarity of $(0,0)$ using the KKT conditions, but by using QE we managed to overcome this, see Figure~\ref{fig:constrfail}.
\end{example}


As we have seen in Example~\ref{example_1D}, we can directly compute the normal cone mapping.

\begin{algorithm}
\caption{Computing $\hat{N}_C(\bar{x})$ directly}\label{algo_normalFull}
\begin{algorithmic}
\Require $C$ semi-algebraic\\\vspace{-0mm}
$S \leftarrow$ QE applied to $\forall \epsilon>0\,\exists\delta>0\,\forall x\in \{y\in C \setminus \{\bar{x}\}\, :\, ||y-\bar{x}||^2\le \delta^2\}\ \ \langle v^*,x-\bar{x}\rangle^2\le \epsilon^2||x-\bar{x}||^2\lor\langle v^*,x-\bar{x}\rangle\le 0$ \\
\Return $\bar{x} \in C \land S$ 
\end{algorithmic}
\end{algorithm}

Example~\ref{example_1D} is the very example of applying Algorithm~\ref{algo_normalFull}, but as was noted there, the quantifier elimination yields some solutions irrelevant to the original problem. It was also noted in that example that the calculations can be simplified if we consider the $\bar{x}\in C$ condition of Definition~\ref{Def_original} while doing the QE calculations. 

\begin{algorithm}[h]
\caption{Computing $\hat{N}_C(\bar{x})$ directly with $\bar{x}\in C$ in direct consideration}\label{algo_normalFull2}
\begin{algorithmic}
\Require $C$ semi-algebraic\\\vspace{-0mm}
\Return QE applied to $\bar{x} \in C \land \forall \epsilon>0\,\exists\delta>0\,\forall x\in \{y\in C\setminus \{\bar{x}\}\, :\, ||y-\bar{x}||^2\le \delta^2\}\ \ \langle v^*,x-\bar{x}\rangle^2\le \epsilon^2||x-\bar{x}||^2\lor\langle v^*,x-\bar{x}\rangle\le 0$ \\
\end{algorithmic}
\end{algorithm}

%
%
%

The quantifiers of the problem in Algorithm~\ref{algo_normalFull2} are free of $\bar{x}$, Hence, the conditions, especially the equational constraints that $\bar{x}\in C$ imposes, can be used to simplify the QE problem required to find the normal cone. We give a step-by-step example of this idea in action here.

\begin{example}\label{ex:3} Calculate the normal cone $N_C(\bar{x})$ subject to $C:= \{(x,y) : xy=0\land x\geq0\land y\geq 0\}$. Let $\bar{x}=(X,Y)$, then using QE we get that $v^* = (v_1,v_2) \in \hat{N}_C (\bar{x})$ if
\begin{align*}
&[XY=0\land X\geq 0 \land Y\geq 0]\land\,\forall \epsilon>0\, \exists \delta>0\, \forall x,y\, [x y=0\land x\geq 0\land y\geq 0\land {(x-X)^2+(y-Y)^2}<\delta^2\\
&\hspace{1cm}\Rightarrow\left((v_1 (x-X)+v_2 (y-Y))^2<\epsilon ((x-X)^2+(y-Y)^2)\lor v_1 (x-X)+v_2 (y-Y)\leq 0\right)]
\intertext{is satisfied. We can simplify the conditions on $\bar{x} \in C$ to get}
&=[(X\geq0\land Y=0) \lor (X=0\land Y\geq 0)]\land\\ &\hspace{1cm}\forall \epsilon>0\, \exists \delta>0\, \forall x,y\, [x y=0\land x\geq 0\land y\geq 0\land {(x-X)^2+(y-Y)^2}<\delta^2\\
&\hspace{1cm}\Rightarrow\left((v_1 (x-X)+v_2 (y-Y))^2<\epsilon ((x-X)^2+(y-Y)^2)\lor v_1 (x-X)+v_2 (y-Y)\leq 0\right)].
\intertext{We then distribute the disjunction over}
&=[X\geq0\land \underline{Y=0} \land
\forall \epsilon>0\, \exists \delta>0\, \forall x,y\, [x y=0\land x\geq 0\land y\geq 0\land {(x-X)^2+(y-Y)^2}<\delta^2\\
&\hspace{1cm}\Rightarrow\left((v_1 (x-X)+v_2 (y-Y))^2<\epsilon ((x-X)^2+(y-Y)^2)\lor v_1 (x-X)+v_2 (y-Y)\leq 0\right)] \\
&\lor[\underline{X=0}\land Y\geq0 \land
\forall \epsilon>0\, \exists \delta>0\, \forall x,y\, [x y=0\land x\geq 0\land y\geq 0\land {(x-X)^2+(y-Y)^2}<\delta^2\\
&\hspace{1cm}\Rightarrow\left((v_1 (x-X)+v_2 (y-Y))^2<\epsilon ((x-X)^2+(y-Y)^2)\lor v_1 (x-X)+v_2 (y-Y)\leq 0\right)].
\intertext{Now, one can apply the underlined equational constraints outside of the quantified formulas in the quantified formula and lower its complexity. Another way of seeing this is to write the prenex form of the formula first and then to use the equational constraints:}
&=[X\geq0 \land Y=0\land
\forall \epsilon>0\, \exists \delta>0\, \forall x,y\, [x y=0\land x\geq 0\land y\geq 0\land {(x-X)^2+y^2}<\delta^2\\
&\hspace{1cm}\Rightarrow\left((v_1 (x-X)+v_2 y)^2<\epsilon ((x-X)^2+y^2)\lor v_1 (x-X)+v_2 y\leq 0\right)] \\
&\lor [ X=0 \land Y\geq0 \land
\forall \epsilon>0\, \exists \delta>0\, \forall x,y\, [x y=0\land x\geq 0\land y\geq 0\land {x^2+(y-Y)^2}<\delta^2\\
&\hspace{1cm}\Rightarrow\left((v_1 x+v_2 (y-Y))^2<\epsilon (x^2+(y-Y)^2)\lor v_1 x+v_2 (y-Y)\leq 0\right)]. 
\intertext{Applying QE to the quantified clauses we get}
&= [X\geq 0 \land Y=0 \land (X < 0 \lor (X = 0 \land v1 \leq 0 \land v_2 \leq 0) \lor (X > 0 \land v_1 = 0))]\\ &\hspace{2cm}\lor [X=0 \land Y\geq 0 \land (Y < 0 \land (Y = 0 \land v_1 \leq 0 \land v_2 \leq 0) \lor (Y > 0 \land v_2 = 0))]\\
&= (X = 0 \land Y = 0 \land v_1 \leq 0 \land 
       v_2 \leq 0) \lor (X = 0 \land Y > 0 \land v_2=0) \lor (X>0 \land Y=0 \land v_1=0).
\end{align*}
The normal cone at $(X,Y) \in C$ is then the closure of the third quadrant if the point is at the origin, and a vertical or a horizontal line if the point is on the positive $x$ or $y$ axis, respectively. This is demonstrated in Figure~\ref{fig_firstquadrant}.
\end{example}

\begin{figure}[h]
\includegraphics{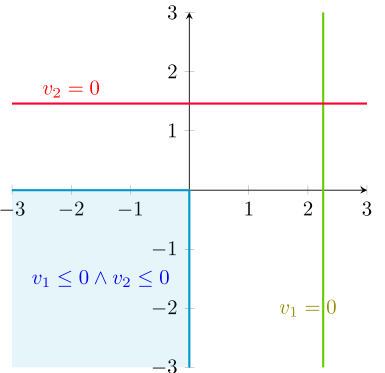}
    \caption{The graph of the  regular normal cone mapping for $C:=[xy=0 \land x\geq 0\land y\geq 0].$}
    \label{fig_firstquadrant}
\end{figure}

Example~\ref{ex:3} showcases how we can split and use the $\bar{x}\in C$ condition to our benefit. This can be done on a broader scale. We can first calculate a CAD of $C$ and then apply QE to individual cells, $C_i$s, of the CAD. Calculating a decomposition of $C$ via CAD involves the original $n$ variables corresponding to the dimensions of the problem and is a much smaller question than calculating the whole normal cone at a fixed point which requires us to do QE on $2n+2$ variables, where the ordering of the variables are also restricted due to the quantified structure.


An important note is that, each $\hat{N}_{C_i}(\bar{x})$ are upper estimates to the normal cone $\hat{N}_C(\bar{x})$. The normal cone calculations can be thought of as finding the restrictions on $v^*$ vectors that satisfy \eqref{eq:def3} that must be satisfied on $C$. By finding these restrictions at parts of $C$, we get weaker conditions, which are upper estimates of the normal cone. The advantage of this approach is that we can feed any found restrictions back into the QE system to lower the complexity of the problem to be solved. The conjunction of the restrictions found over $C_i$ will be the whole set of restrictions on $v^*$ and hence would yield the conditions of the normal cone. 

We now give the basic algorithm of how we will calculate the normal cone, followed by the theorem making sure that the output of the algorithm is, the desired normal cone, $\hat{N}_C(\bar{x})$.


\begin{algorithm}[h]
\caption{Computing $\hat{N}_C(\bar{x})$ at a fixed $\bar{x}\in C$ with CAD considerations}\label{algo_CAD}
\begin{algorithmic}
\Require $C$ semi-algebraic\vspace{-0mm}
\State $CAD \leftarrow$ Cylindrical Algebraic Decomposition of $C$
\State $m \leftarrow$ number of cells of $CAD$
\State Order $CAD$ increasingly over the dimension of the cells (name these cells $C_i$).
\State $\mathcal{N}_0 \leftarrow \mathbb{R}^n$
\For{$i=1\dots m$} 
$\mathcal{N}_i \leftarrow \mathcal{N}_{i-1}\cap \hat{N}_{C_i} (\bar{x})$ calculated with QE
\EndFor\\
\Return $\mathcal{N}_m$
\end{algorithmic}
\end{algorithm}


\begin{theorem}\label{thm:algo4isture}
The output of Algorithm~\ref{algo_CAD}, 
$\mathcal{N}_m$, is equal to $\hat{N}_{C}(\bar{x})$.
\end{theorem}
\begin{proof}
By construction
    \begin{equation}\label{eq:Ni}
        \mathcal{N}_m =\bigcap_{j=1}^m \hat{N}_{C_j}(\bar{x})
    \end{equation}
    holds.
 We fix $\varepsilon>0$ and get for each $C_j$ a $\delta_j$ such that 
    \begin{equation*}
       \forall x\in \left(B_{\delta_j}\cap C_j\right)\setminus \{\bar{x}\} :\quad\langle v^*,x-\bar{x}\rangle^2\le \epsilon^2||x-\bar{x}||^2\lor\langle v^*,x-\bar{x}\rangle\le 0
    \end{equation*} holds.
 We take the smallest among the  $\delta=\min \delta_j$ then the inequality holds for $\delta$ in every $C_j$ and consequently also for the union:
     \begin{equation*}
       \forall x\in \left(B_{\delta}\cap \bigcup_j C_j\right)\setminus \{\bar{x}\} :\quad\langle v^*,x-\bar{x}\rangle^2\le \epsilon^2||x-\bar{x}||^2\lor\langle v^*,x-\bar{x}\rangle\le 0.
    \end{equation*} 
Now as $C=\bigsqcup_{j=1}^m C_j$, the finite disjoint union of its cells, this is just the definition of $\hat{N}_{C}(\bar{x})$.
\end{proof}

There are at least two reasons to start from the lower dimensional cells and move our way up. Firstly, CAD has doubly exponential worst-case complexity in number of variables, and lower dimensional cases allow us to lower this complexity by utilizing equational constraints. Secondly, all the learned conditions on the normal cones can be used in the next iteration of this algorithm, which will result in a lower dimensional search space of the normal cone. Intuitively, this approach lowers the number of parameters the quantifier elimination needs to deal with at each step. We demonstrate this on our running example:

\begin{example} Let's focus on \eqref{conditionsKKT}. We would like to calculate the normal cone at the origin, where the KKT conditions failed. We can apply CAD to \[C:= [x\geq 0 \land (y+x^2)(y-x^2)=0] = [x\geq 0 \land \left(\,(y+x^2)=0 \lor (y-x^2)=0\, \right) ] \] to split in three cells \[
C_1:= (0,0),\ C_2:= [x>0\land (y-x^2)=0],\ C_3:= [x>0\land (y+x^2)=0].\] These cells can be seen in Figure~\ref{fig:constrfail} as the arms of the two black curves and the point of their intersection at the origin.

The calculations of $\hat{N}_{C_i}(\bar{x})$ for points in the cells $C_2$ and $C_3$ can be dealt with using Lagrangians. However, as pointed out in Section~\ref{sec:intro}, the point $(0,0)$ requires special treatment, see \cite[Chapter 3.2.2]{mandlmayr2019disjunctive}. We will focus on this point. Let $\bar{x}=(0,0)$ in \eqref{eq:def3} and apply QE to calculate $\hat{N}_C(\bar{x})$ in succession, as outlined in Algorithm~\ref{algo_CAD}. This yields $\mathcal{N}_1 = \mathbb{R}^2$, $\mathcal{N}_2 = \{(v_1,v_2) : v_1 \leq 0\}$, and $\mathcal{N}_3 = \{(v_1,v_2) : v_1 \leq 0\}$. Theorem~\ref{thm:algo4isture} shows that $\hat{N}_C(0,0) = \{(v_1,v_2) : v_1 \leq 0\}$.

\end{example}


To demonstrate this, we calculate a normal cone calculation that appeared in a real-world 6-dimensional optimization problem on Coulomb friction \cite{Mandlmayretal, MandlmayrThesis} next. Before doing so, we discuss an optimization of Algorithm~\ref{algo_CAD} using Satisfiability Modulo Theories (SMT) approach.

Let $C$ be a semi-algebraic set and $C_i$'s be its CAD cells in some order. Recall, that each $\hat{N}_{C_i}(\bar{x})$ is an upper-estimate of $\hat{N}_{C_i}(\bar{x})$, and the $\mathcal{N}_i$ is the intersection of the first $i$ upper-estimates. We calculate each $\hat{N}_{C_i}(\bar{x})$ using QE, which can be the bottleneck since QE can be computationally expensive. We can instead ask a computationally cheaper intermediate question: ``Can $\mathcal{N}_{i+1}$ be smaller than $\mathcal{N}_i$?" and only calculate QE when the test question yields a positive answer.
The only difference between the sets $\mathcal{N}_{i+1}$ and $\mathcal{N}_i$, defined by \eqref{eq:Ni}, is $\hat{N}_{C_{i+1}}(\bar{x})$. By definition \eqref{eq:def3}, if $v^*\in \hat{N}_{C_{i+1}}(\bar{x})$, then it satisfies
\[ \forall \epsilon>0\, \exists\delta>0\, \forall x\in \{y\in C_{i+1}\setminus \{\bar{x}\}\; : \;||y-\bar{x}||^2\le \delta^2\}\ \ \langle v^*,x-\bar{x}\rangle^2\le \epsilon^2||x-\bar{x}||^2\lor\langle v^*,x-\bar{x}\rangle\le 0.
\]
Now, we ask if there is possibly any point outside of $\hat{N}_{C_{i+1}}(\bar{x})$, that is in $\mathcal{N}_i$. In other words, asking whether $\mathcal{N}_{i+1}$ will be a closer upper-estimate to $\hat{N}_{C}(\bar{x})$ than $\mathcal{N}_{i}$. We can test this by checking if there exist any $(v^*, \epsilon, \delta, x)$, such that 
\begin{equation}\label{eq:SMT}v^* \in \mathcal{N}_i\setminus \lnot[\langle v^*,x-\bar{x}\rangle^2\le \epsilon^2||x-\bar{x}||^2\lor\langle v^*,x-\bar{x}\rangle\le 0] \land \epsilon>0 \land \delta>0 \land 0<|| x-\bar{x}||^2 < \delta^2. \end{equation}
This is a purely existential problem and can be attempted by SMT techniques, such as NLSAT \cite{NLSAT}. The negated clause in the formula \eqref{eq:SMT} is the group of conditions of $\hat{N}_{C_{i+1}}(\bar{x})$. By finding a point that satisfies this clause, we would learn that there may be points outside of $\hat{N}_{C_{i+1}}(\bar{x})$ that are inside of $\mathcal{N}_i$. 

Finding a satisfying assignment of \eqref{eq:SMT} does not guarantee that the original definition of $\hat{N}_{C_{i+1}}(\bar{x})$ with the quantifiers will be satisfied. However, if \eqref{eq:SMT} is unsatisfiable then it is easy to conclude that there are no points outside of $\hat{N}_{C_{i+1}}(\bar{x})$ that is inside of $\mathcal{N}_i$. In other words, the cell $C_{i+1}$ cannot contribute a new restriction to $\mathcal{N}_i$. Then we can directly conclude that $\mathcal{N}_{i+1} = \mathcal{N}_i$. Hence, we can update Algorithm~\ref{algo_CAD} with this optimization, and present Algorithm~\ref{algo_CAD_SMT}.

\begin{algorithm}[h]
\caption{Computing $\hat{N}_C(\bar{x})$ at a fixed $\bar{x}\in C$ with CAD and SMT considerations}\label{algo_CAD_SMT}
\begin{algorithmic}
\Require $C$ semi-algebraic\vspace{-0mm}
\State $CAD \leftarrow$ Cylindrical Algebraic Decomposition of $C$
\State $m \leftarrow$ number of cells of $CAD$
\State Order $CAD$ increasingly over the dimension of the cells (name these cells $C_i$)
\State $\mathcal{N}_0 \leftarrow \mathbb{R}^n$
\For{$i=1\dots m$}
\If{$i\geq 2$} 
\If{There exists no $(v^*, \epsilon, \delta, x)$ satisying \eqref{eq:SMT}}
\State $\mathcal{N}_{i} \leftarrow \mathcal{N}_{i-1}$.
\EndIf
\Else
\State $\mathcal{N}_i \leftarrow \mathcal{N}_{i-1}\cap \hat{N}_{C_i} (\bar{x})$ calculated with QE
\EndIf
\EndFor\\
\Return $\mathcal{N}_m$
\end{algorithmic}
\end{algorithm}

Another optimization that one can do is the boundary mapping. For a CAD cell that has a non-empty intersection with its boundary, one can first check \eqref{eq:SMT} on the boundary (when there are extra equational constraints) and update the set of restrictions $\mathcal{N}_i$ accordingly. Then \eqref{eq:SMT} can be checked again for the interior of the cell with the updated $\mathcal{N}_i$.

With these optimizations in place, calculating normal cones at fixed points can be done effectively. To demonstrate,  we discuss the normal cone calculations for the Coulomb Friction Model \cite{Mandlmayretal} in detail.

\begin{example}\label{ex:5}
Here we calculate $\hat{N}_C(\bar{0})$, where $\bar{0}=(0,0,0,0,0,0)$ and
  \begin{align}\nonumber C:=&  x_4\leq 0\land x_3\geq 0\land \left((x_3=0\land x_4\leq 0)\lor (x_3>0\land x_4=0)\right)\\
   \label{def:C}
    & \land\left(\left(x_1^2+x_2^2=0\land x_5^2+x_6^2\leq x_4^2\right)\lor
   \left(x_1^2+x_2^2\neq 0\land x_5=-\frac{x_4 x_1}{\sqrt{x_1^2+x_2^2}}\land x_6=-\frac{x_4
   x_2}{\sqrt{x_1^2+x_2^2}}\right)\right).
\end{align}

In order to do so we will follow the lines of Algorithm \ref{algo_CAD_SMT} with the following practical adaptions.
 \begin{itemize}
     \item To avoid implementing a CAD cell ordering function that orders cells $C_i$ with respect to their dimension, we placed time limits on the QE problem to be solved. This is a proxy for the cell dimensions as lower dimensional cells generate simpler QE problems that can be solved within the given time limits. If the QE terminates in the given time (a single second for the example) successfully, we update the estimate and retry for the unsuccessful ones after the update.
     \item As mentioned in the note after Algorithm~\ref{algo_CAD_SMT}, for the cells where the QE takes more than the time limit after all considerations, we check if our unresolved cell $C_j$ can contribute something new to the estimate by first checking the satisfiability of \eqref{eq:SMT} on its boundary and then in its interior. Checking the conditions on the boundary restricts the QE problem to a lower dimensional set, and increases the chances of termination within the time limit.
\end{itemize}


To find all the set description of the points $v^*=(v_1,v_2,v_3,\V,g_1,g_2) \in\hat{N}_C(\bar{0})$, we start by  by first calculating the CAD of $C$ in the variable order $x_1\succ x_2\succ x_3 \succ x_4 \succ x_5 \succ x_6$. This yields 22 cells, once calculated with \emph{Mathematica}. (If one would like to calculate this in another computer algebra language, they would need to clear the rational functions and the roots first.) We will refer to these cells as $C_i$s where $i\in\{1,2,\dots,22\}$. 

We first try to calculate individual $\hat{N}_{C_i}(0)$s using QE with a (harsh) time limit of 1 second. We can increase this termination limit if need be. Each $\hat{N}_{C_i}(0)$ is an upper estimate of $\hat{N}_C(0)$ and any point $v^*\in \hat{N}_C(0)$ must also satisfy the conditions of $\hat{N}_{C_i}(0)$s. The first run of QE through the cylindrical cells results in 12 QE resolutions. These resolved cases correspond to lower dimensional cells.

Not all the QE resolutions teach us a new condition for $v^*$. Some cells do not contribute a restriction for $v^*$ at all. For example, we have the 0-dimensional cell $\bar{0}$. The QE for this cell returns the outcome True, indicating that $\bar{0}\in \hat{N}_{C_i}(\bar{0})$ and that $\hat{N}_{\bar{0}}(\bar{0})=\mathbb{R}^6$. Nevertheless, after the first round of calculations, we get to combine all the conditions that are satisfied by the cells of $C$ and get conditions for the first upper estimate $\mathcal{N}_1$ of the normal cone $\hat{N}_C(\bar{0})$: \[\mathcal{N}_1 = [ v_1 = 0 \land v_2 = 0 \land v_3\leq 0 \land g_1 + \V \geq 0 \land g_1 \leq \V \land g_1 + \V \geq 0 \land g_2 \leq \V].\] Conditions $v_1 =0 \land v_2=0$, already lowers the dimension of the problem to a 4-dimensional problem moving forward. 

We can impose the conditions of $\mathcal{N}_1$ on the QE problems we would like to solve for the remaining 10 cells (as in Example~\ref{ex:3}) followed by another round of time-limited QE calculations.  This leads to the elimination of all but one of the remaining cases, and we do not learn any new conditions that refine $\mathcal{N}_1$.

Once again, we note that we can check if the upper estimates reached the desired normal cone using a bottom-up approach, such as NLSAT \cite{NLSAT}. We can check if it is possible to find a point $P:=(v_1,v_2,v_3,\V,g_1,g_2,\epsilon,\delta,x_1,x_2,x_3,x_4,x_5,x_6)$ with $\epsilon,\delta>0$ that satisfies the conditions \eqref{eq:SMT} (before doing any expensive QE calculations), i.e. a point that satisfies conditions of $\mathcal{N}_i$ that does not satisfy the clauses of $\hat{N}_C(\bar{0})$. Effectively, this is asking if there is a possible point in the upper-estimate $\mathcal{N}_i$ that is extraneous. For example, let \[C^* = \left[x_1=0\land x_2 = 0 \land x_3 = 0 \land x_4 < 0 \land x_6^2 + x_5^2 \leq x_4^2 \right], \] which is the last unresolved cell of the normal cone calculations. The QE problem one needs to solve to calculate the normal cone to this cell $\hat{N}_{C^*}(\bar{0})$ is 
 \begin{align}\nonumber\forall\epsilon\,\exists\delta\,\forall x_4,x_5,x_6\, &\left[ \left(\delta >0 \land (x_4\geq0 \lor  ({x_4^2 + x_5^2 +x_6^2})^2 \geq \delta^2  \lor (g_1 x_5 + g_2 x_6 + x_4 \V)^2 < \epsilon^2 ({x_4^2 + x_5^2 +x_6^2})  \right.\right.\\ \label{eq:QEC*} &\left.\left. \lor x_5^2 +x_6^2 > x_4^2)\lor  g_1 x_5 +g_2 x_6 +x_4 \V \leq 0 \right)\lor\epsilon \leq 0 \right]
 \end{align} in prenex form. We can ask if there is a point $P$ that satisfies the conditions of the upper estimate $\mathcal{N}_1$ but does not satisfy the innermost clause of $\hat{N}_{C^*}(\bar{0})$, subject to $\epsilon>0$ and $\delta>0$. Solving this, we see there is such a point: $(0,0,-1,1,-13/16,-15/16, 5/256, 2,171/10,-103/5,-15/2,-1,-1/2,-11/16)$ that has this property. This suggests that the upper estimate $\mathcal{N}_1$ may still be an upper estimate and we may not have discovered all the conditions of $\hat{N}_C(\bar{0})$.

Quantifier elimination problem \eqref{eq:QEC*} does not terminate in a reasonable (24 hours) time. Instead, we focus on a subset of this cell, its boundary: \[\partial C^* :=\left[ x_1=0\land x_2 = 0 \land x_3 = 0 \land x_4 < 0 \land x_6^2 + x_5^2 = x_4^2\right].\] Performing QE to calculate $\hat{N}_{\partial C^*}(\bar 0)$ is much more manageable and the calculations terminate under a minute with the output \[\hat{N}_{\partial C^*}(\bar 0) = \left[\V\geq0 \land \V^2\geq g_1^2 \land \V^2\geq g_1^2+ g_2^2\right].\] Note that $\hat{N}_{\partial C^*}(\bar 0)$ does not impose any restrictions on $v_1,\, v_2,$ and $v_3$.

We can update our upper estimate for $\hat{N}_{C}(\bar 0)$ with $\mathcal{N}_2 = \mathcal{N}_1 \cap \hat{N}_{\partial C^*}(\bar 0)$, which is \begin{equation}\label{eq:CoulombNormalCone}v_1 = 0 \land v_2 = 0 \land v_3\leq 0 \land \V\geq0 \land \V^2\geq g_1^2 \land \V^2\geq g_1^2+ g_2^2.\end{equation} Moreover, we can check that $(v_1,v_2,v_3,\V,g_1,g_2)=(0,0,-1,17/8,-1,-2)$ satisfies $\mathcal{N}_1$ but not $\mathcal{N}_2$. This way ensuring that $ \mathcal{N}_2 \left(\,\supset\hat{N}_C(\bar0) \right) $ is a closer upper estimate to $\hat{N}_C(\bar0)$ than $\mathcal{N}_1$.

Now, the only region we have not checked for conditions of $\hat{N}_C(\bar0)$ is the interior of $C^*$, let's call it $C^o$. $C^o$ is the same as $C^*$ except for the last inequality has to be a strict inequality. Then the QE problem associated with the normal cone of $C^o$ is almost the same as \eqref{eq:QEC*} with $x_5^2+x_6^2 > x_4^2$ replaced by $x_5^2+x_6^2 \geq x_4^2$. We can once again check if there are any points $P$ that can satisfy the conditions of $\mathcal{N}_2$ that is outside of $\hat{N}_{C^o}(\bar0)$. This time SMT proves that no such point $P$ exists. This is equivalent of saying that $\mathcal{N}_2 \subset \hat{N}_{C^o}(\bar0)$. Recall that $C^o$ is the last cell to consider, i.e. we have $\hat{N}_C(\bar0) = \mathcal{N}_2\cap \hat{N}_{C^o}(\bar0)$. Combining the last two statements yields that $\mathcal{N}_2 = \hat{N}_C(\bar0)$ as in \eqref{eq:CoulombNormalCone}.

\end{example}

\section{Other Applications to Semismooth* Newton Method}\label{sec:otherApplications}

In \cite{Mandlmayretal}, the authors take multiple steps to show that a map satisfies subspace containing derivative semismooth* property. One of these steps is showing this property by proving that the graph of this set is the projection of a constructed, more complicated set $P$ with 2 extra parameters, \cite[Proposition 6.3]{Mandlmayretal}, which has the semismooth* property. In fact, they use that the set $P$ is semialgebraic, which implies that it is semismooth*. Constructing such general sets might not be easy but the projection claim can easily be checked for semismooth sets using CAD/SMT. 

The set $C$ defined in \eqref{def:C} is related to the $\mathcal{F}=1$ case of the mentioned graph at the point $\bar{0}$. For a fixed given point $(\underbrace{\bar{v}_1,\bar{v}_2,\bar{v}_3}_{\bar{v}},\bar{\V}, \underbrace{\bar{g}_1,\bar{g}_2}_{\bar{g}})$, there we have \[P := \left\{ (\underbrace{v_1,v_2,v_3}_v,\V,\underbrace{g_1,g_2}_g, \underbrace{p_1,p_2}_p)\ : \begin{array}{l}
   ||v-\bar{v}||^2 + ||(\V,g)-(\bar{\V},\bar{g})||^2 \leq 1,\ ||p||^2<1, \\
   (v_1^2+v_2^2) p_1^2 = v1^2,\    (v_1^2+v_2^2) p_2^2 = v2^2, \\
   v_1p_1\geq 0,\ v_2p_2\geq 0,\ g_1=-\mathcal{F}\V p_1,\ g_2 = -\mathcal{F}\V p_2\\
   v_3\geq 0 ,\ \V\leq 0,\ v_3\V=0
\end{array} \right\}\]

In \cite{Mandlmayretal} , authors proved that the projection of $P$ on $p_1$ and $p_2$ with necessary substitutions is equal to $C\cap B_1(0,0)$. We can verify this in a matter of seconds. One only needs to check if there are any points in $P$'s projection on the variables $p_1$ and $p_2$ (with $(\bar{v},\bar{\V}, \bar{g})=\bar0$ and $\mathcal{F}=1$) that are not in $C$ in \eqref{def:C} (with $(x_1,x_2,x_3,x_4,x_5,x_6)=({v_1,v_2,v_3},\V,{g_1,g_2})$). Similarly, one can check if there are any points in $C\cap B_1(0,0)$ that are not in $P$'s projection. 




\section{outlook}\label{sec:out}
A natural extension of this work would be to incorporate the procedures discussed in this paper into optimization algorithms. There are at least two immediate possibilities.

Firstly, we observed that the places where the constraint qualifications fail were always present in the CAD of the domain. In our examples, these were 0-dimensional cells in the CAD. We expect this property to carry over regardless of the dimension of the issue. These degenerate situations cannot happen when the gradients of the constraint mapping have full rank. Therefore, we suggest to using CAD as a pre-processing for classic optimization methods and check if the lower dimensional cells are associated with a degenerate situation. We especially suggest checking the 0-dimensional cells of the CAD for optimality via applying Algorithm \ref{alg:cap} to the negative gradients at this point. 

Secondly, Algorithm \ref{alg:cap} can be used as an alternative stopping criterion for a numerical optimization algorithm. These algorithms have dedicated stopping criterion, usually fulfilling the KKT conditions. In the situation when the algorithm cannot make any progress, e.g. because a minimum is reached but the KKT conditions are not satisfied, we can resort to Algoritm~\ref{alg:cap} as a secondary stopping criteria. We believe that it will be particularly interesting for problems with a complementarity structure.

The key difficulty for developing a semismooth* Newton method is computing the (regular) co-derivatives, a fully automated procedure will make this research area far more accessible. When we apply our methods to calculate the normal cones for a problem, we get the co-derivatives and the associated linearization virtually for free. For example, if Algorithm~\ref{algo_normalFull2} terminates, we get everything about the optimization problem or generalized equations problem one can wish for. What we get out of the algorithms is not only the solution to a single problem but the general information and method to solve all associated problems.
Another possibility for further research is to first compute the tangent cone via QE and then by polarization obtain the regular normal cone.

Moreover, there are many other abstract objects in variations analysis, like convex conjugate functions, convex closures, polar sets (Definition \ref{def:polar}), and tangent cones. Being able to automatically compute them can greatly ease the work of people involved in theoretical optimization. 

We plan to work on these problems in the near future.

\section{Acknowledgement}

The authors thank Josef Schicho for initiating this collaboration and for his helpful comments. The authors would also like to thank James H. Davenport and Christoph Koutschan for their comments on the manuscript.
The authors would also like to thank Michael Winkler and Matus Benko for their great comments and inputs, from the optimization perspective.

The second author would like to thank the EPSRC grant number EP/T015713/1 and the FWF grant P-34501N for partially supporting his research.





\bibliographystyle{plain}
\bibliography{sample}

\begin{thebibliography}{10}

\bibitem{Abrahametal2020b}
E.~\'Abrah\'am, J.H. Davenport, M.~England, G.~Kremer, and Z.P. Tonks.
\newblock {New Opportunities for the Formal Proof of Computational Real
  Geometry?}
\newblock {\em $SC^2$'20: Fifth International Workshop on Satisfiability
  Checking and Symbolic Computation CEUR Workshop Proceedings}, 2752:178--188,
  2020.

\bibitem{anai2003convex}
H.~Anai and P.A. Parrilo.
\newblock Convex quantifier elimination for semidefinite programming.
\newblock In {\em Proceedings of the International Workshop on Computer Algebra
  in Scientific Computing, CASC}, volume 2003. Citeseer, 2003.

\bibitem{balas1979disjunctive}
E~Balas.
\newblock Disjunctive programming.
\newblock {\em Annals of discrete mathematics}, 5:3--51, 1979.

\bibitem{benko2016numerical}
M.~Benko.
\newblock {\em Numerical methods for mathematical programs with disjunctive
  constraints}.
\newblock PhD thesis, Johannes Kepler University, 2016.

\bibitem{disjunctive}
M.~Benko and H.~Gfrerer.
\newblock New verifiable stationarity concepts for a class of mathematical
  programs with disjunctive constraints.
\newblock {\em Optimization}, 67(1):1--23, 2018.

\bibitem{Bradfordetal2021a}
R.J. Bradford, J.H. Davenport, M.~England, A.~Sadeghimanesh, and A.~Uncu.
\newblock {The DEWCAD Project: Pushing Back the Doubly Exponential Wall of
  Cylindrical Algebraic Decomposition}.
\newblock {\em ACM Comm. Computer Algebra}, 55(3):107--111, 2021.

\bibitem{Brown2001c}
C.W. Brown.
\newblock {The McCallum projection, lifting, and order-invariance}.
\newblock Technical Report MOTS2001.1, United States Naval Academy, 2001.

\bibitem{Brown2005b}
C.W. Brown.
\newblock {On Quantifier Elimination by Virtual Term Substitution}.
\newblock Technical Report USNA-CS-TR-2005-07, 2005.

\bibitem{BrownDavenport2007}
C.W. Brown and J.H. Davenport.
\newblock {The Complexity of Quantifier Elimination and Cylindrical Algebraic
  Decomposition}.
\newblock In C.W. Brown, editor, {\em Proceedings ISSAC 2007}, pages 54--60,
  2007.

\bibitem{brown2020enhancements}
C.W Brown and S.~McCallum.
\newblock Enhancements to lazard’s method for cylindrical algebraic
  decomposition.
\newblock In {\em Computer Algebra in Scientific Computing: 22nd International
  Workshop, CASC 2020, Linz, Austria, September 14--18, 2020, Proceedings 22},
  pages 129--149. Springer, 2020.

\bibitem{Collins1975}
G.E. Collins.
\newblock {Quantifier Elimination for Real Closed Fields by Cylindrical
  Algebraic Decomposition}.
\newblock In {\em Proceedings 2nd. GI Conference Automata Theory \& Formal
  Languages}, pages 134--183, 1975.

\bibitem{DavenportHeintz1988a}
J.H. Davenport and J.~Heintz.
\newblock {Real Quantifier Elimination is Doubly Exponential}.
\newblock {\em J. Symbolic Comp.}, 5:29--35, 1988.

\bibitem{DavenportTonksUncu2021}
J.H. Davenport, Z.P. Tonks, and A.K. Uncu.
\newblock A combined vts/lazard quantifier elimination method.
\newblock 2021.

\bibitem{DavenportTonksUncu2023}
J.H. Davenport, Z.P. Tonks, and A.K. Uncu.
\newblock A poly-algorithmic approach to quantifier elimination.
\newblock {\em arXiv preprint arXiv:2302.06814, accepted SYNASC 2023}, 2023.

\bibitem{kaka}
D.~Davis, D.~Drusvyatskiy, S.~Kakade, and J.D. Lee.
\newblock Stochastic subgradient method converges on tame functions.
\newblock {\em Found. Comput. Math.}, 20(1):119--154, 2020.

\bibitem{NLSAT}
L.~De~Moura and D.~Jovanovi{\'c}.
\newblock A model-constructing satisfiability calculus.
\newblock In {\em International Workshop on Verification, Model Checking, and
  Abstract Interpretation}, pages 1--12. Springer, 2013.

\bibitem{FleKanOut07}
M.L. Flegel, C.~Kanzow, and J.~Outrata.
\newblock Optimality conditions for disjunctive programs with application to
  mathematical programs with equilibrium constraints.
\newblock {\em Set-Valued Analysis}, 15:139--162, 2007.

\bibitem{Gfr14a}
H.~Gfrerer.
\newblock Optimality conditions for disjunctive programs based on generalized
  differentiation with application to mathematical programs with equilibrium
  constraints.
\newblock {\em SIAM Journal on Optimization}, 24(2):898--931, 2014.

\bibitem{Mandlmayretal}
H.~Gfrerer, M.~Mandlmayr, J.~V. Outrata, and J.~Valdman.
\newblock On the {S}{C}{D} semismooth* {N}ewton method for generalized
  equations with application to a class of static contact problems with coulomb
  friction.
\newblock {\em Computational Optimization and Applications}, Nov 2022.

\bibitem{gfrerer2021semismooth}
H.~Gfrerer and J.V. Outrata.
\newblock On a semismooth* newton method for solving generalized equations.
\newblock {\em SIAM Journal on Optimization}, 31(1):489--517, 2021.

\bibitem{gfrerer2022local}
H.~Gfrerer and J.V. Outrata.
\newblock On (local) analysis of multifunctions via subspaces contained in
  graphs of generalized derivatives.
\newblock {\em Journal of Mathematical Analysis and Applications},
  508(2):125895, 2022.

\bibitem{gfrerer2022application}
H.~Gfrerer, J.V. Outrata, and J.~Valdman.
\newblock On the application of the {S}{C}{D} semismooth* {N}ewton method to
  variational inequalities of the second kind.
\newblock {\em Set-Valued and Variational Analysis}, 30(4):1453--1484, 2022.

\bibitem{Ioakimidis2019a}
N.I. Ioakimidis.
\newblock {Sharp bounds based on quantifier elimination in truss and other
  applied mechanics problems with uncertain, interval forces/loads and other
  parameters}.
\newblock Technical Report TR-2019-Q7 University of Patras, 2019.

\bibitem{Io09}
A.D. Ioffe.
\newblock An invitation to tame optimization.
\newblock {\em SIAM J. Optim.}, 19(4):1894--1917, 2008.

\bibitem{Io07}
A.D. Ioffe.
\newblock Variational analysis of regular mappings.
\newblock {\em Springer Monographs in Mathematics. Springer, Cham}, 2017.

\bibitem{lazard1994improved}
D.~Lazard.
\newblock An improved projection for cylindrical algebraic decomposition.
\newblock In {\em Algebraic geometry and its applications: collections of
  papers from Shreeram S. Abhyankar’s 60th birthday conference}, pages
  467--476. Springer, 1994.

\bibitem{MandlmayrThesis}
M.~Mandlmayr.
\newblock {\em {Semismooth* {N}ewton methods for quasi-variational inequalities
  and contact problems with friction}}.
\newblock PhD thesis, Johannes Kepler University, 2022.

\bibitem{mandlmayr2019disjunctive}
M.~Mandlmayr.
\newblock Disjunctive programming in applications, MSc thesis, Johannes Kepler
  University, 2019.

\bibitem{Mo18}
B.S. Mordukhovich.
\newblock {\em Variational analysis and applications}, volume~30.
\newblock Springer, 2018.

\bibitem{Mulliganetal2018b}
C.B. Mulligan, J.H. Davenport, and M.~England.
\newblock {TheoryGuru: A Mathematica Package to Apply Quantifier Elimination
  Technology to Economics}.
\newblock In J.H. Davenport, M.~Kauers, G.~Labahn, and J.~Urban, editors, {\em
  Proceedings Mathematical Software --- ICMS 2018}, pages 369--378, 2018.

\bibitem{RahkooySturm2021c}
H.~Rahkooy and T.~Sturm.
\newblock {Parametric Toricity of Steady State Varieties of Reaction Networks}.
\newblock In {\em Proceedings CASC 2021: Computer Algebra in Scientific
  Computing}, pages 314--333, 2021.

\bibitem{RockWets98}
R.T. Rockafellar and R.~J.-B. Wets.
\newblock {\em Variational Analysis}.
\newblock Springer Verlag, Heidelberg, Berlin, New York, 1998.

\bibitem{RostSadeghimanesh2021b}
A.~Sadeghimanesh and M.~England.
\newblock Polynomial superlevel set representation of the multistationarity
  region of chemical reaction networks.
\newblock {\em BMC bioinformatics}, 23(1):1--26, 2022.

\bibitem{St03}
O.~Stein.
\newblock {\em Bi-level strategies in semi-infinite programming}, volume~71.
\newblock Springer Science \& Business Media, 2003.

\bibitem{Wadaetal2016a}
Y.~Wada, T.~Matsuzaki, A.~Terui, and N.H. Arai.
\newblock {An Automated Deduction and Its Implementation for Solving Problem of
  Sequence at University Entrance Examination}.
\newblock In {\em Proceedings ICMS 2016}, pages 82--92, 2016.

\bibitem{Wilsonetal2013a}
D.J. Wilson, R.J. Bradford, J.H. Davenport, and M.~England.
\newblock {A ``Piano Movers'' Problem Reformulated}.
\newblock Technical Report CSBU-2013-03 Department of Computer Science
  University of Bath, 2013.

\end{thebibliography}

\end{document}